\documentclass[11pt]{article}
\usepackage{graphicx,graphics}

\usepackage{color}
\usepackage{epsfig}
\usepackage{latexsym}
\usepackage{caption}
\usepackage{amsfonts}
\usepackage{amsmath,amsthm,enumerate, mathrsfs, amssymb, indentfirst}
\usepackage{enumerate}
\usepackage{enumitem}
\usepackage{authblk}
\usepackage{pict2e} 
\usepackage{hyperref}

\usepackage{pgfplots}
\pgfplotsset{%
    ,compat=1.16
    ,colormap={mygray}{rgb255(0cm)=(225,225,225); rgb255(1cm)=(255,255,255)}
    }
\usepgfplotslibrary{colormaps, external}
\usepackage{float}
\usepackage{tikz}
\usepackage{cite}
\usetikzlibrary{backgrounds}
\usepackage{rotating}
\usetikzlibrary{positioning}
\usepackage{pgffor}

\newtheorem{theorem}{Theorem}
\newtheorem{conjecture}[theorem]{Conjecture}
\newtheorem{lemma}[theorem]{Lemma}
\newtheorem{corollary}[theorem]{Corollary}
\newtheorem{observation}[theorem]{Observation}
\newtheorem{proposition}[theorem]{Proposition}

\newtheorem{definition}[theorem]{Definition}

\begin{document}
\baselineskip 0.6cm

\title{Signed projective cubes, a homomorphism point of view}

\author[1]{Meirun Chen}  
\affil[1]{School of Mathematics and Statistics, Xiamen University of Technology,\linebreak Xiamen Fujian 361024, China. E-mail: mrchen@xmut.edu.cn, \vspace{2mm}}

\author[2]{Reza Naserasr}
\affil[2]{Université Paris Cité, CNRS, IRIF, F-75013, Paris, France.	
	E-mail: reza@irif.fr.}
\author[3]{Alessandra Sarti}
\affil[3]{UMR 7348 du CNRS, Université de Poitiers, 11 bd
	Marie et Pierre Curie, 86073 POITIERS, France. alessandra.sarti@cnrs.fr}

\date{}
\maketitle

\begin{center}
\begin{minipage}{150mm}

{\bf Abstract}
\noindent The (signed) projective cubes, as a special class of graphs closely related to the hypercubes, are on the crossroad of geometry, algebra, discrete mathematics and linear algebra. Defined as Cayley graphs on binary groups, they represent basic linear dependencies. Capturing the four-color theorem as a homomorphism target they show how mapping of discrete objects, namely graphs, may  relate to special mappings of plane to projective spaces of higher dimensions. 

\noindent In this work, viewed as a signed graph, first we present a number of equivalent definitions each of which leads to a different development. In particular, the new notion of common product of signed graphs is introduced which captures both Cartesian and tensor products of graphs.

\noindent We then have a look at some of their homomorphism  properties. We first introduce an inverse technique for the basic no-homomorphism lemma, using which we show that every signed projective cube is of circular chromatic number 4. Then observing that the 4-color theorem is about mapping planar graphs into signed projective cube of dimension 2, we study some conjectures in extension of 4CT. Toward a better understanding of these conjectures we present the notion of extended double cover as a key operation in formulating the conjectures. As a particular corollary of our result we build a highly symmetric triangle-free graph on 12 vertices with the property that every planar graph of odd-girth at least 7 maps to it. 

With a deeper look into connection between some of these graphs and algebraic geometry, we discover that projective cube of dimension 4, widely known as the Clebsh graph, but also known as Greenwood-Gleason graph, is the intersection graph of the 16 straight lines of an algebraic surface known as Segre surface, which is a Del Pezzo surface of degree 4. We note that an algebraic surface known as the Clebsch surface is one of the most symmetric presentations of a cubic surface. Recall that each smooth cubic surface contains 27 lines. Hence, from hereafter, we believe, a proper name for this graph should be Segre graph.

\end{minipage}
\end{center}

\section{Introduction}\label{sec:Intro}
\setlength{\parindent}{0pt}

Built from hypercubes, \emph{signed projective cubes} play a key role in some of the most central problems in graph theory, such as the four-color theorem and several extensions of it. In this work we have an in depth look at the structure of these signed graphs and some of its homomorphism properties.
The study leads to several new notions that may shed lights to some of these problems. Notable concepts to be studied in this paper are: 1. the notion of ``common product'' which is defined on signed graphs and captures both Cartesian and categorical product. 2. Extended double cover of signed graphs.

We begin with establishing the basic notation and terminologies. A \emph{signed graph} $(G, \sigma)$ is a graph $G$ together with an assignment $\sigma$ of signs ($+$ or $-$) to the edges of $G$. A \emph{switching} of $(G, \sigma)$ at a cut $(X, V(G)\setminus X)$ is to multiply signs of the edges of the cut to $-$. Sign of a substructure of $(G, \sigma)$ is the product of the signs of its edges considering multiplicity. With respect to sign and parity of the length, there are four distinguished types of closed-walks and cycles in $(G, \sigma)$. Given a signed graph $(G,\sigma)$ and for each $ij$ in $\mathbb{Z}_2^2$ we define $g_{ij}(G, \sigma)$ to be the length of a shortest closed walk whose parity of the number of negative edges is $i$ and parity of the length is $j$. The \emph{unbalanced girth} (also known as the \emph{negative girth}) of $(G, \sigma)$, denoted $g_{-}(G,\sigma)$, is the length of its shortest negative cycle, setting it to $\infty$ if there is no such a cycle. It follows from the definitions that $g_{-}(G, \sigma)=\min \{g_{10}(G,\sigma), g_{11}(G, \sigma)\}$.

A \emph{homomorphism} of a signed graph $(G, \sigma)$ to a signed graph $(H,\pi)$ is a mapping $\varphi$ of $V(G)$ to $V(H)$ and $E(G)$ to $E(H)$ such that adjacencies and incidences, as well as signs of closed walks are preserved. The condition is equivalent to a switching $(G, \sigma')$ of $(G, \sigma)$ and a mapping that preserves not only adjacencies and incidences but also signs of the edges with respect to $\sigma'$ and $\pi$. We refer to \cite{NSZ21} for more details.

Following the definition we have a basic no-homomorphism lemma.

\begin{lemma}[The no-homomorphism lemma]\label{lem:No-Hom}
If $(G,\sigma)\to (H,\pi)$, then for each $ij \in \mathbb{Z}_2^2$ we have $g_{ij}(G,\sigma)\geq g_{ij}(H,\pi)$.
\end{lemma}

Signed graphs considered in this work are simple in the sense that between each pair of vertices there is at most one edge of a given sign. That implies that there could be two edges connecting a given pair of vertices, one of each sign, in which case the subgraph induced by the two vertices is referred to as \emph{digon}. One may also allow one loop of each sign on a vertex, but in this work negative loops will mostly be forbidden. 
 
\subsection{Signed Cayley graphs}

As an extension of the notion of Cayley graphs we define \emph{signed Cayley graphs}.

Let $\Gamma$ be an Abelian group and let $S^+$ and $S^-$ be two subsets of $\Gamma$ with property that $-S^+=S^+$ and $-S^-=S^-$. The signed Cayley graph $(\Gamma,S^+, S^-)$ is defined to be the signed graph whose vertices are the elements of $\Gamma$ where two vertices are adjacent by a positive edge if and only if $x-y\in S^{+}$ and they are adjacent by a negative edge if and only if $x-y\in S^-$. In particular, if $0\in S^+$, then each vertex has a positive loop on it and similarly a negative loop if $0\in S^-$. The sets $S^-$ and $S^+$ are called \emph{difference sets}.

\begin{observation}\label{vertex-transitive}
 Any signed Cayley graph $(\Gamma,S^+, S^-)$ is vertex-transitive.
 \end{observation}
 \begin{proof}
  For given vertices $x$ and $y$, the mapping $f(u)=u+y-x$ is an automorphism
that maps $x$ to $y$.
 \end{proof}

When $\Gamma$ is a binary group, the signed Cayley graph $(\Gamma,S^+, S^-)$ will be referred as a \emph{signed binary Cayley graph} or as a \emph{signed cube-like graph}.

A special class of signed cube-like graphs are the signed projective cubes which are the central part of this work. They are defined in numerous equivalent ways in the next section. 

\section{Signed Projective Cubes}\label{sec:SPC}

In this section signed projective cubes are defined in several different ways. 
Each definition, given in a separate subsection, gives a new insight to these graphs. The first definition is the one from which we have derived the name. A historical note is given at the end of the section.

\subsection{As projections of the hypercubes}

Recall that hypercube of dimension $n$, denoted $H(n)$, is the Cayley graph $(\mathbb{Z}_2^n, \{e_1, e_2, \ldots, e_n\})$ where $e_i$'s are the standard basis. A pair of vertices are said to be antipodal if they are at distance $n$, i.e., they differ at any coordinate.

\begin{definition}\label{def:AsTheProjectiveCube}
The \emph{projective cube}  of dimension $k$, denoted $PC(k)$, is the graph obtained from the hypercube of dimension $k+1$ by identifying antipodal pairs of vertices. The signed projective cube of dimension $k$, denoted $SPC(k)$, is then obtained from $PC(k)$ by assigning a positive sign to the edges that correspond to one of the first $k$ coordinates of $H_{k+1}$ and a negative sign to those that correspond to the last coordinate of $H_{k+1}$, that is the coordinate which is eliminated.
\end{definition}

One may view $H_n$ as the skeleton of the standard unit cube in $\mathbb{R}^n$ with $O$ at the center of it. With such a view, when building projective space from $\mathbb{R}^{k+1}$ the image of $H_{k+1}$ is the graph $PC(k)$. Thus, in some sense, one may view projective cubes as unite cubes in projective spaces.

\subsection{As augmented cubes} 

An alternative definition, but also based on hypercubes is as follows.

\begin{definition}\label{def:As/augmentedHypercube}
The signed projective cube of dimension $k$ is the graph built from the hypercube of dimension $k$ all whose edges are positive by adding a negative edge between each pair of antipodal vertices. 
\end{definition}

To observe that the two definitions are equivalent, consider an inductive definition of $H(k+1)$.  In such a definition, $H(k+1)$ is built from two disjoint copies $H_1$ and $H_2$ of $H(k)$ by adding a perfect matching that connects corresponding pairs of vertices from two disjoint copies.
To find the antipodal of a vertex $x_1$  in $H_1$ part of $H(k+1)$, we first must find $\overline {x}_1$, the antipodal of $x_1$ in $H_1$, then $\overline {x}_2$, the  match of $\overline {x}_1$ in $H_2$, is the antipodal of $x_1$ in $H(k+1)$. This mapping of vertices of $H_1$ to their antipodals in $H_2$ is also an isomorphism of $H_1$ to $H_2$. Thus, when the vertices of $H_2$ are projected onto their antipodal vertices in $H_1$, the edges of $H_2$ are mapped to the edges of $H_1$. Therefore, the resulting graph of this projection is the graph built on $H_1$ where the only new edges are the images of the matching between $H_1$ and $H_2$ which are precisely the edges that are given a negative sign. Such images connect exactly the  antipodal pairs of $H_1$.

\subsection{As Cayley graphs}

Definition~\ref{def:AsTheProjectiveCube} in turn invokes an algebraic definition of the projective cubes. Noting that the antipodal of a vertex $x$ in the hypercube $H(k)$ is the vertex $x+J$ (where $J$ is the all-1 vector) we have the following equivalent definition of the signed projective cubes.

\begin{definition}\label{def:AsCayleyGraphs}
The signed projective cube of dimension $k$ is the Cayley graph $(\mathbb{Z}_{2}^k, \{e_1,e_2, \ldots, e_k\}, \{ J \})$ where the $e_i$'s are the standard basis and $J$ is the all-1 vector.
\end{definition}

\begin{figure}[ht]
\centering 

\begin{minipage}{1\textwidth}
\centering
\begin{tikzpicture}
\foreach \i in {1,...,16}
{
\draw[rotate=-22.5*(\i-3)] (3.2, 3.2) node[circle, fill=white, draw=black!80, inner sep=0mm, minimum size=2mm]  (\i){};
}

\foreach \i/\j in {1/16, 2/3, 4/13, 5/12, 6/7, 8/9, 10/11, 14/15}
{
    
\draw[thick, blue] (\i) -- (\j);
}

\foreach \i/\j in {1/14, 2/13, 3/4, 5/10, 6/9, 7/8, 11/12, 15/16}
{
    
\draw[thick, teal] (\i) -- (\j);
}

\foreach \i/\j in {1/8,2/11,3/10,4/5,6/15,7/14,9/16,12/13}
{
    
\draw[thick, brown] (\i) -- (\j);
}

\foreach \i/\j in {1/2,3/16,4/15,5/6,7/12,8/11,9/10,13/14}
{
    
\draw[thick, purple] (\i) -- (\j);
}

\foreach \i/\j in { 1,2,3,4,9,10,11,12}
{
    
\draw[thick, red] (\i) -- (\the\numexpr\i+4);
}

\end{tikzpicture}
\caption{An edge-colored presentation of $PC(4)$: $(\mathbb{Z}_2^4, ({\color{brown} e_1}, {\color{teal} e_2}, {\color{blue} e_3}, {\color{purple} e_4}, {\color{red} J}))$\textsuperscript{\tiny 1}}
\label{fig:CirculantPC(4)}
\end{minipage}
\end{figure}
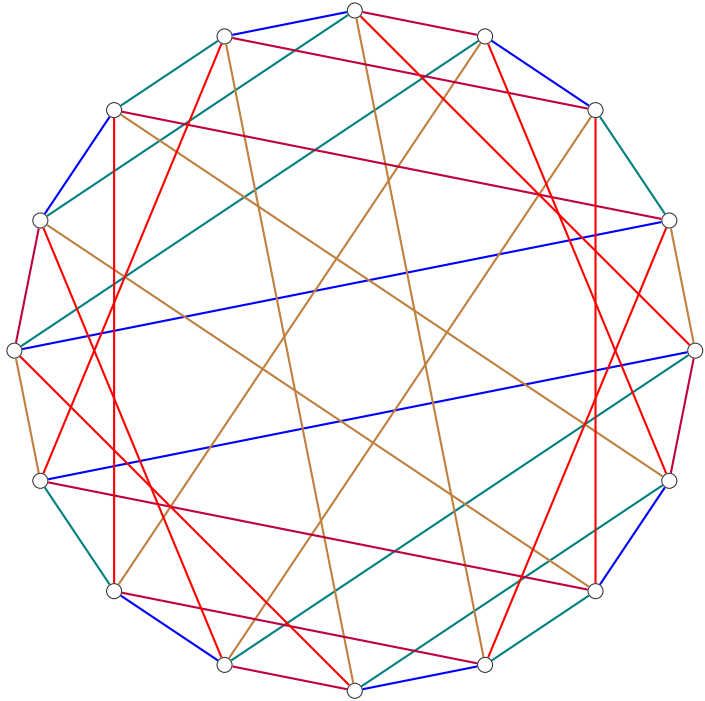

\footnotetext[1]{A labeling of vertcies starting at the top and in the clockwise orientation:\\ $\;  0000, 1000,1100,1110,1111,0111,0011,0001,0101,1101,1001,1011,1010,0010,0110,0100$}

This definition leads to a natural edge-coloring of $PC(k)$, where each edge is colored with its corresponding element from  $\{e_1,e_2, \ldots, e_k, J \}$. Observe that, since we are working on a binary group, this is indeed a proper edge-coloring, i.e., edges incident to the same vertex receive distinct colors. See Figure~\ref{fig:CirculantPC(4)} for an edge-colored presentation of $PC(4)$ (commonly known as the Clebsch graph). It is important to note that, in this definition, edges corresponding to $J$ are not really different from those corresponding to $e_i$'s. Indeed, signed projective cubes are highly symmetric as we will see later. To see this basic symmetry, in the following proposition, we give a Cayley graph presentation of $SPC(k)$ in a more general setting. 

\begin{proposition}
	Let $S=\{s_1,s_2,\ldots, s_k\}$ be a linearly independent subset of $\mathbb{Z}_{2}^n$ and let $s^*=s_1+s_2+\cdots+s_k$. Then the signed Cayley graph $(\mathbb{Z}_{2}^n, S, \{s^*\})$ consists of $2^{n-k-1}$  isomorphic connected components, each isomorphic to $SPC(k)$.
\end{proposition}

In particular, if we consider the set of vectors in  $\mathbb{Z}_{2}^k$ with exactly two nonzero coordinates which are also (cyclically) consecutive, and choose an arbitrary one as $s^*$ while putting the rest in $S$, we obtain two disjoint copies of $SPC(k)$. From this view it is apparent that the signed graph $SPC(k)$ is edge-transitive with respect to switching-isomorphism.

More generally if we take two sets $S^+$ and $S^-$ of the elements of $\mathbb{Z}_{2}^k$ such that the sum of the elements of the two are equal, that this is the only linear dependency formula in $S^+\cup S^-$, that $S^+\cup S^-$ generates $\mathbb{Z}_{2}^k$, and that $S^-$ is of an odd order, then the signed Cayley graph $(\mathbb{Z}_{2}^k, S^+, S^-)$ is switching equivalent to $SPC(k)$. To observe this, given any two elements $s_1$ and $s_2$ of $S^+\cup S^-$, we define $X_o$ (respectively, $X_e$) to be the set of vertices $u$ where in generating $u$ by $S^+\cup S^-$ only one of $s_1$ or $s_2$ is employed (respectively either both or none of $s_1$ or $s_2$ are employed). Then $(X_o, X_e)$ is a partition of vertices where any edge between the two parts either corresponds to $s_1$ or to $s_2$. Thus a switching at this edge-cut is equivalent to repositioning $s_1$ and $s_2$ inside $S^+\cup S^-$, with an emphasis on the fact that both of them must be swapped. In particular, if $S^+$ is of an even order, then all its elements, after a pairing, can be moved to $S^-$, showing that $SPC(2k)$ is switching equivalent to $(PC(2k),-)$ where all edges are negative.

\subsection{As power graphs of cycles}

We first recall the notion of power graph from \cite{BNT15}. Given a graph $G$, let $2^{V(G)}$ be the (binary) group whose elements are the subsets of $V(G)$ and whose operation is the symmetric difference. As defined in \cite{BNT15}, the power graph $G$, denoted $pow(G)$, is the Cayley graph $(2^{V(G)}, E(G))$. Extending this definition, we define the \emph{power graph} of a signed graph $(G, \sigma)$, denoted $pow(G, \sigma)$, to be the signed Cayley graph $(2^{V(G)}, E^+(G, \sigma), E^-(G, \sigma))$.

Observe that, assuming $G$ is connected, $pow(G)$ and $pow(G, \sigma)$ would always consist of two isomorphic components: one induced on the vertices of even order and the other on the vertices of odd order. As mentioned in \cite{BNT15}, the power graph of $P_{n+1}$ consists of two isomorphic copies of the hypercube of dimension $n$. The power graph of $C_{n+1}$ consists of two isomorphic copies of $PC(n)$. Moreover, if we take the signed cycle $C_{-(n+1)}$ instead, then its power graph consists of two isomorphic copies of $SPC(n)$. More precisely:

\begin{definition}\label{def:PowerGraphofC_n}
	Given a signed cycle $C_{-(n+1)}$, the power graph $pow(C_{-(n+1)})$ consists of two isomorphic connected components each of which is a signed projective cube of dimension $n$.
\end{definition}

The isomorphism between this definition and the previous definition is clear by taking the set of vectors with two (cyclically) consecutive 1 as the difference set when we want to apply the definition as a signed Cayley graph. When vertices of $C_{-n-1}$ are labeled by $1,2,\ldots, n+1$ in the cyclic order, then its edges correspond to this set of vectors.

\subsection{Constructed from posets}

When hypercubes are viewed as poset graphs, the projective cubes, built from the projection of hypercubes, find a new definition which provides quite an insightful view to these signed graphs. Recall that in the definition by projection, antipodal vertices of the hypercube $H(k+1)$ are identified in order to form $PC(k)$. 
In the classic definition of the hypercube $H(n)$, each vertex $v$ is a vector in $\mathbb{Z}_2^n$. By associating the set of coordinates at which $v$ is 1, we obtain a presentation of $H(n)$ with vertices as the subsets of an $n$-set where vertices $A$ and $B$ are adjacent if $A\subset B$ and $|A|=|B|-1$. We refer to this as the poset representation of $H(n)$. 
In such a representation, the antipodal of a vertex (of $H(k+1)$) is its complement. Thus, when $PC(k)$ is formed from $H(k+1)$ by identifying antipodal pairs of vertices, each new vertex receives a label of two complementary subsets of the reference $(k+1)$-set (the set which is used to view $H(k+1)$ as a poset). The transition to the signed projective cubes is as follows:

\begin{definition}\label{def:PosetGraph}
	Let $S=S^+ \cup S^-$ be a set of size $k+1$ where $S^+$ is of size $k$ and $S^-$ has a single element. Then $SPC(k)$ is a graph whose vertices are pairs of complementary subsets $\{A, \bar{A}\}$ of $S$ where vertices $\{A, \bar{A}\}$ and $\{B, \bar{B}\}$ are adjacent if the symmetric difference of $A$ and $B$ consists of a single element $s$. Furthermore, the corresponding edge is positive if $s$ is in $S^+$ and it is negative if $s\in S^-$.
\end{definition}

We note that in this definition $\{A, \bar{A}\}$ is an unordered pair. For example if $B$ and $\bar{A}$ have a symmetric difference of size one, then the two vertices are also adjacent.   

The canonical isomorphism with the definition as a Cayley signed graph is provided by setting $S^+=\{e_1,e_2, \ldots, e_{k}\}$ and $S^-=\{J \}$. The condition that $S^-$ has single element can be replaced with ``$S^-$ has an odd number of elements", in which case we get a switching equivalent copy of $SPC(k)$ as discussed before.

In this poset presentation of $PC(k)$, a vertex $\{A, \bar{A} \}$ may be simply represented by the smaller of the two sets $A$ and $\bar{A}$. With such a labeling, the vertices of $PC(k)$ are subsets of order at most $\lfloor \frac{k}{2}\rfloor$ of $S_{k}=\{e_1,e_2, \ldots, e_{k},J \}$. However, to present vertices of $PC(2i-1)$, when $|A|=|\bar{A}|=i$ one must make a choice between $A$ and $\bar{A}$, but indeed any choice would be ok.  

With such a labeling of vertices, the distance between two vertices can be computed by the following formula.

\begin{proposition}\label{prop:ShortestDistInPCk}
	The distance between vertices $A$ and $B$ of the projective cube $PC(k)$ is $\min \{|A \oplus B|, k+1-|A \oplus B|\}$. More precisely, these two numbers correspond to the lengths of shortest positive and negative $A-B$ paths in $SPC(k)$. 
\end{proposition}

Moreover, given a pair $A$ and $B$ of vertices at distance $i+j$ ($d(A,B)=i+j$) the set of vertices at distance $i$ from $A$ and $j$ from $B$ determines $A\oplus B$. This is stated more precisely in the next proposition but taking $A=\emptyset$. For the general choice of $A$, it would suffice to take the symmetric difference with $A$.

\begin{proposition}\label{prop:i+j=dVertices}
Given a vertex $\{B, \bar{B}\}$ ($|B|< |\bar{B}|$) of $PC(k)$ and positive integers $i$ and $j$ such that $i+j=|B|$, the union of vertices (sets) at distance $i$ from $\emptyset$ and $j$ from $B$ is the set $B$. Furthermore, if $|B|= |\bar{B}|$ (thus $k$ being odd), then the internal vertices of all $\tfrac{k+1}{2}$-paths connecting $\emptyset$ and $B$ induce two connected components. The union of vertices at distance $i$, $(i \leq \tfrac{k+1}{2})$, from $\emptyset$ in one component is $B$ and in the other is $\bar{B}$.
\end{proposition}

The poset definition leads to strong connection between (signed) projective cubes and many well known families of graphs. In particular, for even values of $k$, $k=2i$, the set of vertices $\{A, \bar{A}\}$ with $|A|=i$ induces a subgraph isomorphic to a special family of Kneser graphs known as the odd graphs. We leave the proof as an exercise.

\begin{proposition}
	In the poset presentation of $PC(2i)$, the subgraph induced by vertices $\{ A, \bar{A} \}$, $|A|=i$, is isomorphic to the Kneser graph $K(2i+1, i)$.
\end{proposition}

We have observed from the Cayley definition that $SPC(k)$ is vertex transitive. From this poset definition, we gather that the diameter of the graph is $\lfloor \frac{k}{2}\rfloor$. Furthermore, it follows from the previous proposition that in $PC(2i)$ the set of vertices at distance $i$ from a given vertex induces an isomorphic copy of $K(2i+1,i)$. As particular examples we observe that the Petersen graph, isomorphic to $K(5,2)$, is a subgraph of $PC(4)$, and that $K(7,3)$, and in particular the Coxeter graph as a well known subgraph of $K(7,3)$, are subgraphs of $PC(6)$.

\subsection{Inductive definition}\label{sec:InductiveDefinition}

Recall that, as one of the equivalent definitions, the hypercube of dimension $n+1$ is built from two disjoint copies of $H_n$ by adding a perfect matching between corresponding vertices. This has lead to the widely studied notion of the \emph{Cartesian product} of graphs, thus claiming that $H_{n+1}=H_n \Box H_1$.
 
An attempt to define $PC(k)$ inductively is one of the first places that shows the importance of defining signed projective cubes rather than just the projective cubes. It is this attempt that had lead us to the definition of \emph{extended double cover} presented in \cite{NSZ21}. Given a signed graph $(G, \sigma)$, its Extended Double Cover, denoted $EDC(G, \sigma)$, is the signed graph built from $(G, \sigma)$ as follows: for each vertex $x$ of $G$ we have two vertices $x_0$ and $x_1$ connected by a negative edge. For each positive edge $xy$ of $(G, \sigma)$ we have two positive edges: $x_0y_0$ and $x_1y_1$. For each negative edge $xy$ of $(G, \sigma)$ we have two positive edges: $x_0y_1$ and $x_1y_0$. In a geometrical view, a positive edge of $(G, \sigma)$ is viewed as a strip in $EDC(G, \sigma)$ whereas a negative edge is viewed as a twisted strip. To switch at a vertex $x$ of $(G, \sigma)$ is the same as exchange the roles or sides of $x_0$ and $x_1$.   

Having observed that $PC(1)$ is simply the digon, we have the following equivalent definition of the projective cubes.

\begin{definition}\label{def:SPC as EDC}
The signed projective cube of dimension 1, $SPC(1)$, is the digon. For $k\geq 2$, the signed projective cube of dimension $k$ is defined as $SPC(k)=EDC(SPC(k-1))$.
\end{definition}
 
This definition, combined with the above mentioned geometric view, gives a new insight to some of the most well known graphs such as $K_4$, $K_{4,4}$ and the Clebsch graph. These are presented in Figures~\ref{fig:SPC(2)}, \ref{fig:SPC(3)}, \ref{fig:SPC(4)}.

\begin{figure}[ht]
\begin{center}
\resizebox{14cm}{5.5cm}{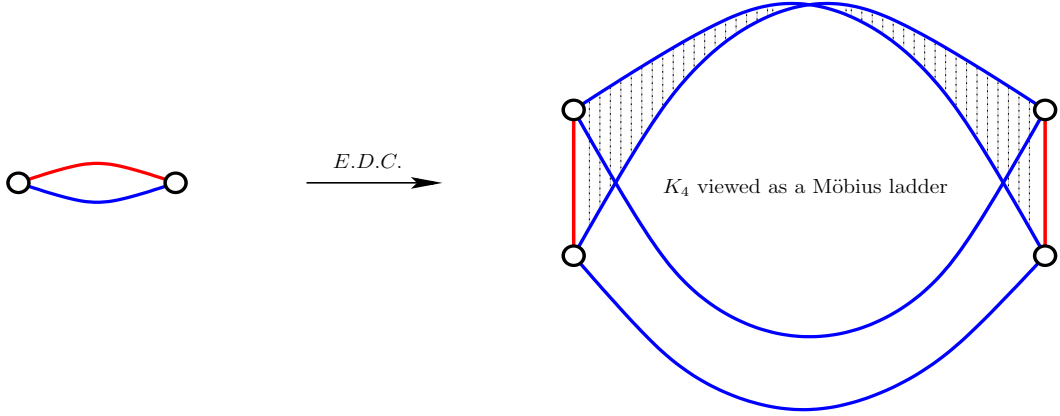 }\end{center}
\caption{$SPC(2)=EDC(SPC(1))$}\label{fig:SPC(2)}
\end{figure}

\begin{figure}[ht]
\begin{center}
\resizebox{14cm}{6.5cm}{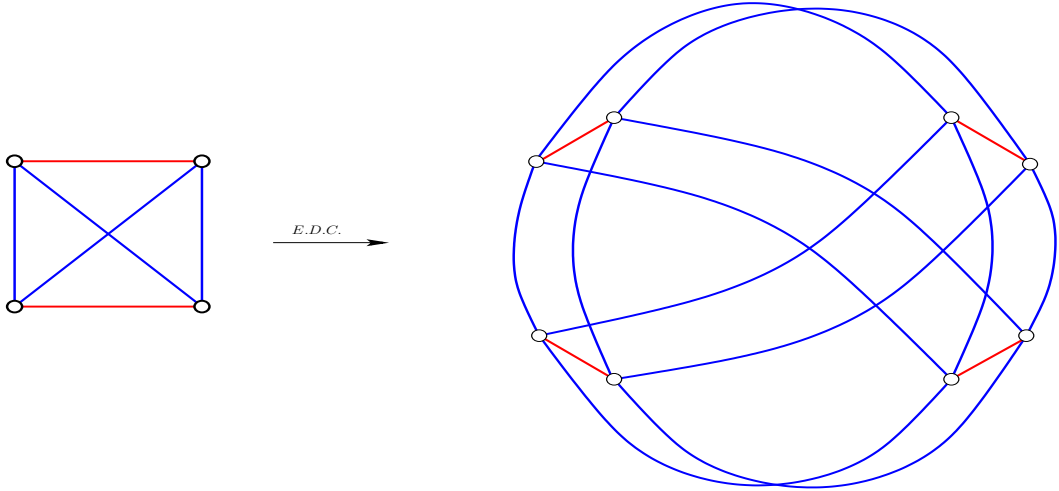 }\end{center}
\caption{Stripes and twists in $K_{4,4}$}\label{fig:SPC(3)}
\end{figure}

\begin{figure}[ht]
\begin{center}
\scalebox{0.4}{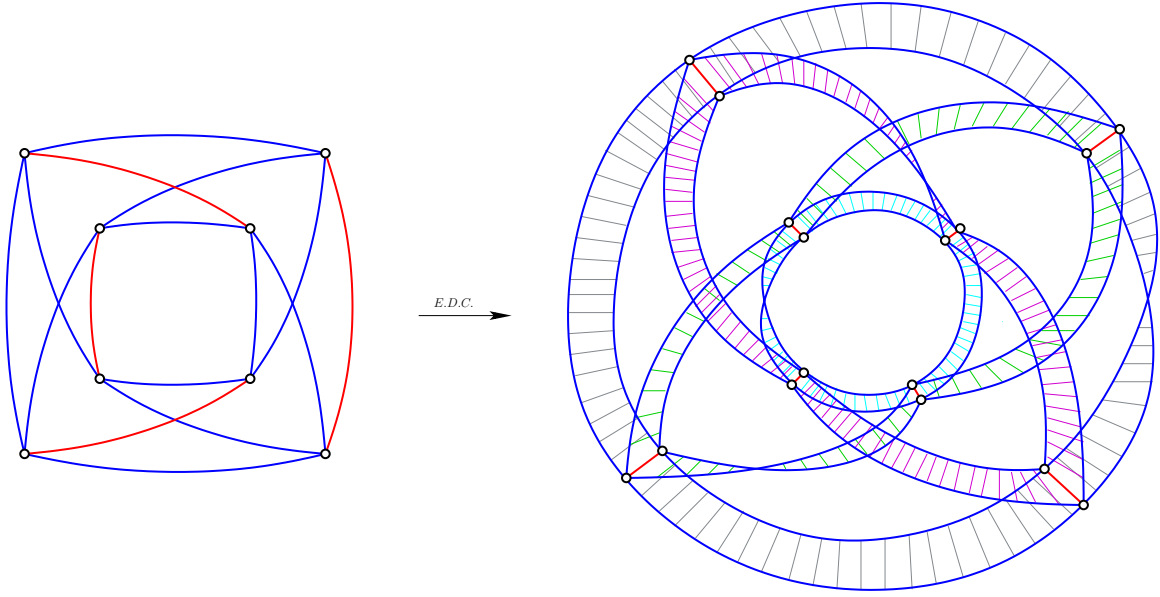 }\end{center}
\caption{Clebsch graph viewed as combinations of M\"obius ladders}\label{fig:SPC(4)}
\end{figure}

%
%
%
%

Let us now prove that $SPC(k)$ as defined in Definition~\ref{def:SPC as EDC} is isomorphic to $(\mathbb{Z}_{2}^k, \{e_1,e_2, \ldots, e_k\}, \{ J \})$. The case $k=1$ is immediate by the choice of $SPC(1)$. We proceed by induction on $k$. Assume $SPC(k-1)$ is isomorphic to $(\mathbb{Z}_{2}^{k-1}, \{e_1,e_2, \ldots, e_{k-1}\}, \{ J \})$, where $J$ is all 1 vector in $\mathbb{Z}_{2}^{k-1}$. With a vertex $x$ of $SPC(k-1)$ viewed as a vector in $\mathbb{Z}_{2}^{k-1}$, we define $x_0$ and $x_1$ to be the vectors in $\mathbb{Z}_{2}^{k}$ obtained from $x$ by adding the $k^{\text{th}}$-coordinate which is valued $0$ in $x_0$ and $1$ in $x_1$. If $x-y=e_i$ in $\mathbb{Z}_{2}^{k-1}$ for some $i$, $1\leq i \leq k-1$, then we have $x_0-y_0=e_i$ and $x_1-y_1=e_i$ in $\mathbb{Z}_{2}^{k}$ and these vertices are connected by positive edges in $EDC(SPC(k-1))$. If $x-y=J$ in $\mathbb{Z}_{2}^{k-1}$, then $x_0-y_1=x_1-y_0=J$ in $\mathbb{Z}_{2}^{k}$. Thus in $EDC(SPC(k-1))$ we have positive edges $x_0y_1$ and $x_1y_0$. Furthermore, with this notation, vertices $x_0$ and $x_1$ whose difference is $e_k$ are connected by a negative edge. So far thus we have shown that $EDC(SPC(k-1))$ is isomorphic to the signed Cayley graph  $(\mathbb{Z}_{2}^k, \{e_1,e_2, \ldots, e_{k-1}, J\}, \{ e_{k} \})$. However, as mentioned before, we can swap $J$ with $e_k$ to get a canonical presentation of the signed graph.

A basic property of extended double cover operation is the following. For a proof we refer to \cite{FNX23}.

\begin{proposition}\label{porp:g_ijEDC}
Given a signed graph $(G, \sigma)$ we have the followings:
\begin{itemize}
\item $g_{01}(EDC(G,\sigma))=g_{01}(G,\sigma)$,
\item $g_{10}(EDC(G,\sigma))=g_{11}(G,\sigma)+1$,
\item $g_{11}(EDC(G,\sigma))=g_{10}(G,\sigma)+1.$
\end{itemize}
\end{proposition} 

\begin{corollary}\label{cor:EDC bip-antiBalance}
If $(G, \sigma)$ is equivalent to $(G, -)$, then $EDC(G, \sigma)$ is a signed bipartite graph whose negative girth is $g_{-}(G, \sigma)+1$. Conversely, if $(G,\sigma)$ is a signed bipartite graph, then there is switching on $EDC(G,\sigma)$ after which all edges are negative and moreover the negative girth of $EDC(G, \sigma)$ which is equal to the odd girth of the underlying graph is $g_{-}(G, \sigma)+1$.
\end{corollary}

\subsection{As a product of signed graphs}

Extending the previous definition, we arrive at a new definition of a product for signed graphs which uniformly extends the Cartesian product and the categorical product of graphs (the latter is also known under several other names such as Tensor product, direct product and Hedetniemi product).

Recall, that in the inductive definition of the hypercube we have $H_{n+1}=H_n\Box H_1$. As a generalization of this, one observes that if $n=k+l$, then  $H_{n}=H_k\Box H_l$. To have an analogue extension to this, we define a new product of two signed graphs as follows:

\begin{definition}
Given two signed graphs $(G, \sigma)$ and $(H, \pi)$ we define their common product, denoted $(G, \sigma) \circ (H, \pi)$, to be a graph whose vertex set is $V(G)\times V(H)$ with the following signed edges. If $xy\in E^+(G,\sigma)$ and $uv \in E^+(H,\pi)$, then $(xu,yu), (xv,yv), (xu, xv)$, and $(yu, yv)$ are all positive edges of $(G, \sigma) \circ (H, \pi)$.   If $xy\in E^-(G,\sigma)$ and $uv \in E^-(H,\pi)$, then $(xu,yv)$, and $(xv,yu)$ are all negative edges of $(G, \sigma) \circ (H, \pi)$.   
\end{definition}

In other words, in the common product of two signed graphs the positive edges are formed from Cartesian product of the two subgraphs induced by positive edges. And the negative edges are formed by the categorical product of the two subgraphs induced by negative edges. Thus, in particular, the Cartesian product of two graphs $G$ and $H$, is the underlying graph of $(G, +) \circ (H, +)$ and the categorical product of two graphs $G$ and $H$, is the underlying graph of $(G, -) \circ (H, -)$. 

While the Cartesian product is the natural extension of viewing $H_n$ as product of two smaller cubes, the following definition of the signed projective cubes plays important role in initiating the definition of common product.

\begin{definition}
The signed projective cube of dimension $1$ is the digon. For $n=k+l$, $k,l\geq 1$, the signed projective cube of dimension $n$ is defined as $SPC(n)=SPC(k)\circ SPC(l)$.
\end{definition}

In particular, the signed projective cube of dimension $n$ can be defined as the common product $SPC(k-1)\circ SPC(1)$. 

To verify that this definition results in the same signed graphs as the previous ones, we consider the definition as augmented cube. As the subgraphs induced by positive edges in $SPC(k)$ and $SPC(l)$ are, respectively, isomorphic to $H_k$ and $H_l$, in the common product the set of positive edges induces the Cartesian product $H_k\Box H_l$ which is isomorphic to $H_n$. In this view of $H_n$, the antipodal of a vertex $v$ of $H_n$ is obtained by taking the antipodal in coordinates corresponding to $H_k$ and the antipodal in coordinates corresponding to $H_l$. The set of negative edges in the common product of $SPC(k)$ and $SPC(l)$ then connects a vertex of $H_n$ to its antipodal.

{\bf Remark} While many concepts in signed graphs are invariant under switching operation, the common product defined here is strongly based on the choice of the signature. A product after a switching on one of the components would normally result in a signed graph where even the underlying graph is substantially changed.

\subsection{Reverse construction}

Finally, we present a construction of $SPC(k)$ from $SPC(k+1)$. This has the advantage of unifying the two notions of minors and homomorphisms and it would help to develop or prove theories on signed projective cubes by induction on the dimension. To achieve this goal, however, we will add a positive loop to each vertex of signed projective cube. Regarding the Cayley definition one may assume that the difference set for the positive edges contains $0$ as well, we use $SPC^{\circ}(k)$ to denote the version with added positive loops. In other words  $SPC^{\circ}(k)=(\mathbb{Z}_{2}^k, \{0, e_1,e_2, \ldots, e_k\}, \{ J \})$. 

\begin{proposition}
	Given a switching equivalent form of the signed projective cube  $(\mathbb{Z}_{2}^{k+1}, S^+, S^-)$ and an element $s\in S^+$, if we identify each pair of vertices connected by a (positive) edge corresponding to $s$, we get a switching equivalent copy of $SPC^{\circ}(k)$.
\end{proposition}

One may view this identification both as a minor operation of signed graphs where some positive edges are contracted, or a homomorphism operation where we allow positive loops. This is of special interest in the study of balanced-chromatic number \cite{JMNNQ23} and its refinement, the circular chromatic number (see Section~\ref{sec:CircularColoring}). 

We should note, to repeat this operation to build $SPC^{\circ}(k-i)$ from $SPC^{\circ}(k)$ by induction on $i$, one should either delete the positive loops first or after the building simplify the resulting signed graph by deleting multi-edges of the same sign.

\section{A historical note}

Special cases of projective cubes are among most well known graphs. The first case $PC(1)$ is the digon, perhaps lesser known because its underlying graphs is not simple, but it is a basic graph in the study of 2-edge-colored or signed graphs. The next case, $SPC(2)$ or rather $PC(2)$ is $K_4$. It is the homomorphism problem to this graph that captures the four-color theorem and as we will see motivates the homomorphism study of $PC(2k)$ and $SPC(k)$ in general. The underlying graph of $SPC(3)$, $PC(3)$ is isomorphic to $K_{4,4}$.  

The graph $PC(4)$ is commonly known as the Clebsch graph and in some sense is one of the first graphs to be studied. The associated name Clebsch perhaps is because of a misunderstanding as we explain next.
An algebraic surface is the zero set of a polynomial $f$ over a (closed) field. The surface associated to a degree 3 polynomial such as $x^3+y^3+z^3+t^3$ is called a cubic surface. Such a surface can be embedded in the complex projective space of dimension 3. Alfred Clebsch proved in 1866 that every smooth cubic surface on an algebraically closed field is rational, the proof starts with a straight line on the surface.

The study of the straight lines on cubic surfaces however goes back to earlier work of Arthur Cayley and George Salmon who showed in 1849 that every smooth cubic surface over an algebraically closed field contains exactly 27 lines (see Section 9 \cite{D12} for more on the geometry of cubic surfaces). The intersection graph of this structure of 27 lines is the complement of what is known as Schläfli graph and sometimes is referred to as such.

This is a highly symmetric 10-regular graph. If one removes a vertex and all its neighbours, that is to remove a line and all that intersects it from the original line configuration, then what we are left with is a 5-regular graph on 16 vertices which is isomorphic to $PC(4)$. This is not the end of the story though. This 16 vertices graph can be directly defined as the intersection of straight lines of another surface, namely Segre quartic surface. This surface is named after Corrado Segre because of his studies in 1884. Segre quartic surface is an intersection of two quadrics in 4 dimensional projective space: These can be described as $$x^2+y^2+z^2+t^2+v^2=0$$ and $$a_1 x^2+a_2 y^2+a_3 z^2+a_4 t^2+a_5 v^2=0$$ where the $a_i$'s are all distinct complex coefficients and nonzero.

For further details we refer to \cite{S1884} and also \cite{S1951}. It contains 16 (straight) lines whose intersection graph is isomorphic to $PC(4)$.


One may also note that cubic surfaces, from which the 27 vertices graph is built, are the Del Pezzo surfaces of degree 3. The Del Pezzo surfaces of degree 4 are exactly Segre quartic surfaces whose line structure gives the graph $PC(4)$. If we similarly remove a vertex and all its neighbours, (that is a line and all those that intersects it in a Segre quartic surface) what remains is isomorphic to one of the most well known graphs, namely the Petersen graph. This graph in turn is isomorphic to the intersection graph of the 10 lines in a Del Pezzo surfaces of degree 5. Finally, removing a vertex and its neighbors from the Petersen graph we have the 6-cycle which is isomorphic to the intersection graph of the six lines of the Del Pezzo surfaces of degree 6.

Recall that Del Pezzo surfaces are all obtained by blowing-up $r$ points in the complex projective plane. In this way one gets a Del Pezzo surface of degree $9-r$ and by studying lines and conics through the $r$ points one can find the number of lines contained in the Del Pezzo surface. In this construction one can easily see the graph properties described before (see the book \cite{D12} and \cite{B78} for more details and properties of Del Pezzo surfaces).


%
%

The projective cube of dimension 4 has surprise appearance on other subject besides algebraic surfaces.
It is one of few known triangle-free strongly regular graphs. But, perhaps more surprisingly, it plays a role in determining the Ramsey number $R(3,3,3)$. To be precise, it is known that $R(3,3,3)=17$, in one direction this identity means that there exists a 3-edge-coloring of $K_{16}$ such that every color class induces a triangle-free graph. Greenwood and Gleason were first to give, in \cite{GG55}, such a 3-edge-coloring of $K_{16}$. It turned out that there are only two such colorings, and each color class in any of these two colorings induces a graph isomorphic to $PC(4)$. The original coloring of Greenwood and Gleason gives yet another definition of $PC(4)$: Vertices are the elements of the field $GF[16]$ (built on $GF[2]$ by an irreducible polynomial of degree four, e.g., $x^4-x-1$). Two vertices are adjacent if their difference is a cubic residue (i.e., $x^3$, $x^3+x^2$, $x^3+x$, $x^3+x^2+x+1$, and $1$). For this reason, the graph is referred to as the Greenwood-Gleason graph in \cite{BM76}.

Some homomorphism properties of this graph will be addressed later in this work.

For the general class of projective cubes, the first appearance seems to be in \cite{M76} where, without naming the graphs and using the notation $\Box_{k}$ for what we have defined as $PC(k)$, Meredith proved that $PC(2k)$ and $H_n$ are the only triple-transitive triangle-free graphs with no $K_{2,3}$ as a subgraph. Following this line of work, $PC(k)$ was named ``folded cube'' in \cite{BCN89}. We prefer the name ``projective cube'' because first of all, the operation using which $PC(k)$ is built from $H_{k+1}$ is the same as projection which is used to build $d$-dimensional projective space from the $(d+1)$-dimensional sphere, and, secondly, the term ``folding'' in homomorphism of graphs is reserved for a special type of homomorphism where one identifies vertices at distance 2. We note that while identifying vertices at distance two resembles a folding, identifying all pairs of vertices at maximum distance, specially with no common center, can hardly be imagined as a folding.

The property of $PC(2k)$ being triple-transitive does not work for $PC(2k+1)$ when they are viewed as graphs, but when viewed as signed graphs and with proper definition, every $SPC(k)$ is proved to be triple-transitive. This will be fully addressed in a forthcoming work.

Noting that $PC(2k-1)$ is bipartite, the chromatic number of $PC(2k)$ was first proved to be 4 in \cite{S87}. Since $PC(2k)$ contains the Kneser graph $K(2k+1, k)$ as an induced subgraph, and that this special family of Kneser graphs are referred to as ``odd graphs", in this reference the term ``extended odd graphs" is used to denote the family $PC(2k)$. For more on the coloring and chromatic properties of signed projective cubes see next sections.  

\section{As subgraph of binary Cayley graphs}\label{sec:Cayley}

A Cayley graph based a binary group is called a \emph{cube-like} graph by L. Lovasz \cite{H75}. This naming is based on the fact that hypercubes are a basic family of such Cayley graphs.
Following the interest in this family of graphs and building on the result of M. Sokolova \cite{S87}, C. Payan, in \cite{P92}, proved a surprising result that no binary Cayley graph has chromatic number 3, i.e. given a binary Cayley graph either it is bipartite or its chromatic number is at least 4. The main step in Payan's proof of this fact is to find a small 4-chromatic subgraph in any cube-like graph of odd-girth $2k+1$. The subgraph he found, turned out to be the well-known generalized Mycielski graph on odd cycles. For a definition of these graphs and presentation of their bipartite analogue we refer to \cite{GNRT23}. We note here that as one of the best examples of graphs of high odd-girth and high chromatic number, generalized Mycielski graphs were independently discovered by other authors.

A strengthening of Payan's result is given in \cite{BNT15} where it is shown that:

\begin{theorem}\label{thm:PC(2k)subgraph}
Any cube-like graph $H$ of odd-girth $2k+1$ contains $PC(2k)$ as an induced subgraph.
\end{theorem} 
   
The conclusion of Payan on the chromatic number of cube-like graphs then is followed from earlier result of Sokolova \cite{S87}. However, Payan's proof find much smaller 4-chromatic subgraphs when the cube-like graph is not bipartite.

Here we generalize this theorem to the larger class of signed binary Cayley graphs. We note that since $SPC(2k)$ is switching equivalent to $(PC(2k), -)$, that is to say one can apply a switching on $SPC(2k)$ to have only negative edges, the homomorphism properties among signed graphs $SPC(2k)$ is the same as the homomorphism properties among the graphs $PC(2k)$. Thus our focus will be on signed binary Cayley graphs where the underlying graph is bipartite, but the proof technique applies generally. In the following section, we generalize the coloring result using the notion of circular chromatic number.

\begin{theorem}
 If a signed bipartite binary Cayley graph $\hat{G}=Cay(\Gamma, S^+, S^-)$ has unbalanced girth $2k$, then it contains the signed projective cube
$SPC(2k-1)$ as an induced subgraph.
 \end{theorem}

\begin{proof}
Let $UC_{2k}$ be an unbalanced cycle of $Cay(\Gamma, S^+, S^-)$ of length $2k$. By Observation \ref{vertex-transitive}, without loss of generality, we assume that $0\in V(UC_{2k})$. Suppose the edges of $UC_{2k}$ correspond to $e_1,\ldots,e_l,e_{l+1},\ldots,e_{2k}$, where $e_1,\ldots,e_l\in S^{+}$, $e_{l+1},\ldots,e_{2k}\in S^{-}$,
$l$ is odd. Let $S=\{e_1, \ldots,e_{2k}\}$. By the definition of a Cayley graph, we have $e_1+e_2+\cdots+e_l+e_{l+1}+\cdots+e_{2k}=0$. But since $2k$ is the negative girth of $G$, this is the only relation among the elements of $S$. In other words, given subsets $A$ and $B$ of $S$, we have  $\displaystyle \sum_{e_i \in A} e_i= \sum_{e_j \in B} e_j$ if and only if either $A=B$ or $A=B^{\complement}$. Because otherwise we have nonempty subset $X$ of $S$ (the symmetric difference of $A$ and $B$) whose sum of elements is $0$. But then sum of the elements of $X^{\complement}$ is also $0$ and one of the two sets would represent a negative cycle of length strictly smaller than $2k$. 

For each pair $(A, A^\complement)$ of the subsets of $S$, let $u_A=\displaystyle \sum_{e_i \in A} e_i$ be the vertex corresponding to $A$ in $\hat{G}$. The previous claim implies that distinct pairs $(A, A^\complement)$ results in distinct vertices $u_A$. Let $V'=\{ u_A \mid A \subset S\}$. We have observed that $|V'|=2^{2k-1}$. We claim that the subgraph induced by $V'$ is isomorphic to $SPC(2k-1)$.   

For this, we use the poset definition of $SPC(2k-1)$. If subsets $A$ and $B$ are such that they differ in one element, say $e_i$ of $S$, then $u_A$ and $u_B$ are connected by an edge labeled $e_i$. On the other hand, we claim that if both pairs $A, B$ and $A, B^\complement$ differ in more than one element, then $u_A$ and $u_B$ are not adjacent in $\hat{G}$. If so, then $u_A-u_B=e'$, but then of the two nontrivial relations  $\displaystyle \sum_{e_i \in A-B} e_i=e'$ and $\displaystyle \sum_{e_i \in A-B^\complement} e_i=e'$ one would correspond to a negative cycle of length strictly smaller than $2k$, contradicting the choice of $2k$ as the negative girth of $\hat{G}$.  \end{proof}

\section{Homomorphisms and circular coloring of signed projective cubes}\label{sec:CircularColoring}

The basic homomorphism relation between signed projective cubes is quite similar to that of homomorphism between negative cycles. More precisely:

\begin{proposition}\label{prop:HomOfSPC}
	For any positive integer $k$, we have $SPC(k+2)\to SPC(k)$. 
\end{proposition} 

\begin{proof}
We use the definition of $SPC(k)=(\mathbb{Z}_2^{k}, \{e_1, \ldots, e_k\}, \{J\})$. We will basically project $SPC(k+2)$ onto its last $k$ coordinates. To this end, we first apply a switch in $SPC(k+2)$ so that edges corresponding to $e_1, e_2$ and $J$ are the negative ones. The projection $\varphi: V(SPC(k+2)) \to V(SPC(k))$ is defined as follows. Given a vertex $u$ of $SPC(k+2)$, if the first two coordinates of $u$ induce $00$ or $11$, then $\varphi$ maps $u$ to the element $u'$ of $V(SPC(k))$ by deleting the first two coordinates. If the first two coordinates of $u$ induce $01$ or $10$, then $\varphi$ maps $u$ to the element $u'+J$ of $V(SPC(k))$.

Let $u$ and $v$ be two adjacent vertices of $SPC(k+2)$. If they differ in $e_1$ or $e_2$, then one of them, say $u$ is mapped to $u'$ and the other, say $v$ is mapped to $v'+J$. Thus $\varphi(u)$ and $\varphi(v)$ are adjacent with a negative edge in $SPC(k)$. If $u-v$ is an element of  $\{e_3, \ldots, e_k\}\cup \{J\}$, then their projection on the first two coordinates is identical, thus the same rule in defining $\varphi$ is applied and hence their images are connected by the same type of edge.	
\end{proof}

By associativity of homomorphism, we conclude that $SPC(k+2i)\to SPC(k)$ for any positive integer $i$. These are then the only type of homomorphism relation in this class of graphs. 

For a smoother inductive approach in dealing with homomorphism problems related to this class of graphs, we may consider $SPC^{\circ}$. By identifying all pairs of adjacent vertices with an edge corresponding to $e_{k+1}$ we have the following.

\begin{proposition}\label{prop:HomOfSPC-o}
	For any positive integer $k$, we have $SPC^{\circ}(k+1)\to SPC^{\circ}(k)$. 
\end{proposition}

The circular chromatic number of a signed graph $(G, \sigma)$, denoted $\chi_c(G,\sigma)$, is defined to be the smallest circumference ($r$) of a circle $C$ for which a mapping $\phi : V(G) \to C$ exists satisfying that for each negative edge $uv$ the distance of $\phi(u)$ and $\phi(v)$ is at least $1$ and for each positive edge $xy$ the distance of $\phi(x)$ and $\phi(y)$ is at most $\frac{r}{2}-1$. For basic facts such as existence of $r$ and that it is a rational number at least $2$ when $(G,\sigma)$ has no negative loop we refer to \cite{NWZ21}, but noting the conditions on positive edges and negative edges are reversed. We should note that for signed graph $(G, -)$ where all edges are negative, $\chi_c(G,-)=\chi_c(G)$ and that adding a positive loop to a vertex does not influence the circular chromatic number.

Circular 4-coloring is of special interest because $SPC^{\circ} (1)$ is the circular clique corresponding to circular 4-coloring. In other words, $\chi_c(G, \sigma)\leq 4$ if and only if $(G, \sigma) \to SPC^{\circ}(1)$. 

The transitivity of homomorphisms implies that if $(G,\sigma)\to (H,\pi)$, then $\chi_c(G, \sigma)\leq \chi_c (H,\pi)$. Thus, noting that for every $k$ we have $SPC(k)\to SPC^{\circ}(1)$, we have $\chi_c(SPC(k)) \leq 4$. Here we show that equality always hold. This fact is implied from some other known result, basically we know some minimal subgraphs whose circular chromatic numbers are 4, see \cite{GNRT23} and references therein for more details. The advantage of the proof here is the proof technique itself which is purely combinatorial and simple, while the other proofs lie on topological concepts such as the winding number. For notions of coloring, such as the chromatic number, the fractional chromatic number and the circular chromatic number which can be defined by homomorphism, the transitivity of homomorphism implies the so called no-homomorphism lemma. That for homomorphism based parameter $\phi$ if $G\to H$, then $\phi(G) \leq \phi(H)$ and that often the inequality is strict. In the proof below we do the opposite: using a homomorphism of $SPC^{\circ}(k)$ to $SPC^{\circ}(k-1)$ we show that if $SPC(k)$ admits an $r$-coloring for some $r<4$, we can reverse engineer to build a circular $r$-coloring of $SPC^{\circ}(k-1)$. Repeating this process, we get a contradiction at $SPC^{\circ}(2)$ or $SPC^{\circ}(1)$.

\begin{theorem}\label{thm:XcSPCk}
	For every positive integer $k$, we have $\chi_c(SPC(k))=4$. 
\end{theorem}

\begin{proof}
	As mentioned, a homomorphism of $SPC(k)$ to $SPC^{\circ}(1)$ provides a circular 4-coloring of $SPC(k)$. It remains to show that there is no circular $r$-coloring for $r<4$. Consider the Cayley definition of $SPC(k)$ as $(\mathbb{Z}_2^k, S^+ , S^-)$ with a nonempty $S^+$ and let $s_1\in S^+$. Toward a contradiction, assume $\phi$ is a circular $r$-coloring of $SPC(k)$ with $C$ being the circle to which vertices of $SPC(k)$ are mapped to. Observe that by contracting edges corresponding to $s_1$ we obtain a copy of $SPC^{\circ}(k-1)$. Given vertices $x$ and $x'$ of $SPC(k)$, connected by an edge corresponding to $s_1$, let $x_1$ be the vertex in the image $SPC^{\circ}(k-1)$ obtained from contracting the (positive) edge $xx'$. We define a mapping $\varphi$ of $V(SPC^{\circ}(k-1))$ to the circle $C$ as follows. Assume that in the clockwise direction $\phi(x)\phi(x')$ is the shorter side of the circle which then is of length at most $\frac{r}{2}-1$. Then define $\varphi(x_1)=\phi(x)$. 
	
	We claim that $\varphi$ must be a circular $r$-coloring of $SPC^{\circ}(k-1)$. Suppose $x_1y_1$ is an edge of $SPC^{\circ}(k-1)$. That means there is an element says $s_2\neq s_1$ of $S^+, S^-$ such that either we have $x-y=s_2$ or $x-y'=s_2$. By the symmetry of $y$ and $y'$ we assume the latter. We assume, moreover, that in the definition of $\varphi$ we have taken $\varphi(x_1)=\phi(x)$. If $\varphi(y_1)=\phi(y')$, then condition on the edge $x_1y_1$ is inherent by the edge $xy'$  of $SPC(k)$ and nothing left to prove. Thus we assume that $\varphi(y_1)=\phi(y)$. In summary, by symmetries we have assumed that $\phi(x)\phi(x')$ and  $\phi(y)\phi(y')$, in this (clockwise) orders, are the shorter sides of the circle when the circle is partitioned to two parts by their two end points. To complete the proof we consider two cases based on whether $s_2\in S^+$ or $s_2\in S^-$.
	

	\begin{itemize}
		\item $s_2\in S^-$. 
		In this case, we claim that $\phi(x)\phi(x')$ and  $\phi(y)\phi(y')$ have no common point. That is because, otherwise, one end point of one of the two intervals belongs to the other. By symmetries assume $\phi(y)$ is on $\phi(x)\phi(x')$, hence $d(\phi(y),\phi(x'))\leq d(\phi(x),\phi(x'))\leq \frac{r}{2}-1<1$. This contradicts the fact that $d(\phi(y),\phi(x'))\geq 1$.  Thus we may assume $\phi(x),\phi(x'),\phi(y),\phi(y')$ are in this order in clockwise direction (see Figure \ref{fig:both outside}). Then one of the two intervals of the circle with end points $\phi(y)$ and $\phi(x)$ contains $\phi(x')\phi(y)$ and the other contains $\phi(y')\phi(x)$ both of which are of length at least 1 because $\phi$ is a circular coloring. Hence we have $d(\varphi(x_1), \varphi(y_1))\geq 1$, as required.
				
		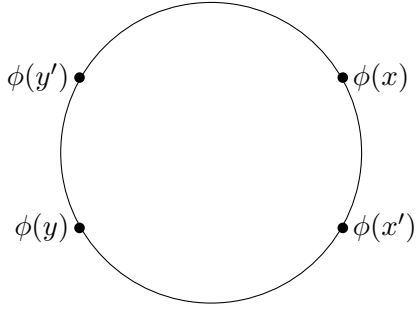
\begin{figure}
			\centering
			
			\begin{tikzpicture}
				\draw (-3,1)ellipse(2 and 2);
				\fill(-1.25,2) circle (2pt);
				\fill(-1.25,0) circle (2pt);
				\fill(-4.75,2) circle (2pt);
				\fill(-4.75,0) circle (2pt);

				\node [right] at (-1.25,2){$\phi(x)$};
			    \node [right] at (-1.25,0){$\phi(x')$};
			    \node [left] at (-4.75,0){$\phi(y)$};
				\node [left] at (-4.75,2){$\phi(y')$};

				

			\end{tikzpicture}
			
			\caption{The only possible ordering of $\phi(x), \phi(x'), \phi(y),\phi(y')$
				\label{fig:both outside}}
			
		\end{figure}
		
	\item  $s_2\in S^+$. 
	First we consider the case when $\phi(x)\phi(x')$ and  $\phi(y)\phi(y')$ have a common point. By symmetries then we may assume that $\phi(y)$ is on $\phi(x)\phi(x')$. But then we have $d(\phi(x),\phi(y))\leq d(\phi(x),\phi(x'))\leq \frac{r}{2}-1$ which is as we wish. 
	
	Next we assume that the two intervals do not intersect. Hence $\phi(x),\phi(x'),\phi(y),\phi(y')$ are in this order in the clockwise direction (Figure~\ref{fig:both outside}). Furthermore, we note that each of the consecutive pair of points (in $\phi(x),\phi(x'),\phi(y),\phi(y')$) correspond to a positive edge of $SPC(k)$ and thus for each of them either the corresponding arc in clockwise direction is of length at most $\frac{r}{2}-1$, or the complement. If the complement for one pair is of smaller length, then any pair of point among the four are at distance at most $\frac{r}{2}-1$ and we are done. So it should be the case that each of the four intervals $\phi(x)\phi(x'), \phi(x')\phi(y), \phi(y)\phi(y')$ and $\phi(y')\phi(x)$ is of length at most $\frac{r}{2}-1$. But since together they cover the full circle, at least one of them is of length at least $\frac{r}{4}$. Hence we have $\frac{r}{4}\leq \frac{r}{2}-1$ which implies $r\geq 4$. This contradicts the assumption that $r<4$. \qedhere
		\end{itemize}
\end{proof}

The proof of Theorem~\ref{thm:XcSPCk} was based on the canonical edge-coloring of $SPC(k)$ where edges are labeled $\{e_1, e_2, \dots, e_k, J\}$. There are two key properties that make the proof work: The first is that, having chosen a label $a \in  \{e_1, \dots, e_k, J\}$ for the set of (positive) edges to be contracted a key property is that if an edge $xy$ labeled $a$ is adjacent to an edge labeled $b$, say $yz$, then two edges are in a 4-cycle, say $xyzt$, where $zt$ is labeled $a$ and $tz$ is labeled $b$.  This ensures that when (all) edges labeled $a$ are contracted and coloring $\phi$ is modified to coloring $\varphi$ as in the proof, the conditions of circular coloring will be enforced on edge connecting vertices formed by contraction. The second property is that after a sequence of contractions where we contract all edges except those labeled $e_k$ and $J$, the resulting signed graph, after removing multiedges of the same sign, is $SPC^{\circ}(1)$ whose circular chromatic number is $4$.

\section{Signed projective cubes and packing signatures}\label{sec:Packing}

Homomorphisms of a signed graph to a signed projective cube captures a notion of packing as stated in the following theorem.

\begin{theorem}\label{thm:SPC(k)-packing}\cite{NRS13}
A signed graph $(G, \sigma)$ maps to $SPC(k)$ if and only if there are signatures $\sigma_1, \sigma_2, \ldots, \sigma_{k+1}$ such that each edge of $G$ is negative in precisely one of $(G, \sigma_i)$'s, $1\leq i \leq k+1$.
\end{theorem}

Observe that for the condition of this theorem to hold, depending on the parity of $k$, the signed graph $(G, \sigma)$ is either equivalent to $(G, -)$ ($k$ even) or the graph $G$ is bipartite ($k$ odd).
The theorem can be generalized to the class of all signed graphs after some modifications.

Let $SPC^{\circ}(k)$ be the signed graph obtained from $SPC(k)$ by adding a positive loop on each vertex. In other words, $SPC^{\circ}(k)$ is the signed Cayley graph $(\mathbb{Z}_{2}^k, \{0, e_1,e_2, \ldots, e_k\}, \{ J \})$. 

Given a signed graph $(G, \sigma)$, the \emph{signature packing number} of $(G, \sigma)$, denoted $p(G,\sigma)$, is defined to be the maximum number $l$ of signatures $\sigma_1, \sigma_2, \ldots, \sigma_l$ each equivalent to $\sigma$ and with the property that each edge of $G$ is negative in at most one of $(G, \sigma_i)$'s, $1\leq i \leq l$.

It is shown in \cite{NY23} that:
\begin{theorem}\label{thm:looped-SPC(k)-packing}
Given a signed graph $(G, \sigma)$, the signature packing number of it is equal to the largest $k$ such that $(G, \sigma)$ maps to $SPC^{\circ}(k-1)$.
\end{theorem}

This theorem generalizes the previous one because when $(G, \sigma)$ is in one of the two mentioned classes and it maps to $SPC^{\circ}(k-1)$, then it also maps to $SPC(k-1)$.

\section{As homomorphism bounds}\label{sec:Hom-Bound}

Based on the fact that $PC(2)$ is isomorphic to $K_4$, and noting that any proper 4-coloring of a graph $G$ is equivalent to a homomorphism to $K_4$, the four-color theorem is viewed as a special case of homomorphism from planar graphs to $PC(2k)$ in \cite{N07}. With $PC(2k+1)$ being a bipartite graph, thus homomorphically being equivalent to $K_2$, $PC(2k+1)$ is of little to no interest in the homomorphism study of graphs. This deficiency led B. Guenin to define homomorphism of signed graphs using which a bipartite analogue of the key Conjecture of \cite{N07} was introduced. The combined conjecture, using the terminologies we have developed here, is as follows.

\begin{conjecture}\label{conj:PlanarToSPC(k)}
Given a signed planar graph $(G, \sigma)$, if $g_{ij}(G,\sigma) \geq g_{ij}(SPC(k))$, then $(G, \sigma) \to SPC(k)$. 
\end{conjecture}

In other words the conjecture suggests that for a planar signed graph to map to $SPC(k)$ the necessary conditions of the no-Homomorphism lemma, Lemma~\ref{lem:No-Hom}, are also sufficient.  We remark that $g_{ij}(SPC(k))= g_{ij}(C_{-k-1})$ and thus the condition in the conjecture can be replaced with $g_{ij}(G,\sigma) \geq g_{ij}(C_{-k-1})$. 

Through the relation to packing signatures, and duality, the conjecture is strongly connected to a number of important conjectures. E.g., to some conjectures on edge-chromatic number and fractional edge-coloring of planar graphs, and to a conjecture on classification of binary clutters. It is through some of these connections, and based on results of \cite{NY23, NY24, DKK16, CES15, CEKS15 } that the conjecture is verified for $k\leq 7$, see \cite{NRS13}. 

The conjecture is also related to other notions such as the fractional chromatic number of planar graphs, their circular chromatic number and so on. We refer to \cite{N13} for more on this subject.

We believe better understanding of the signed projective cubes may help to better understanding of this deep conjecture and, to this end, we present a strengthening of the conjecture using an inductive definition provided here (Section~\ref{sec:InductiveDefinition}).

Given a class $\mathcal{C}$ of signed graphs, we say a signed graph $\hat{B}$ bounds $\mathcal{C}$ if every signed graph in $\mathcal{C}$ admits a homomorphism to $\hat{B}$. Let $\mathcal{SP}_{k}$ be the class of signed planar graphs $(G, \sigma)$ satisfying $g_{ij}(G,\sigma) \geq g_{ij}(C_{-k})$. Observe that for odd values of $k$, each signed graph $(G, \sigma)$ in $\mathcal{SP}_{k}$ is switch-equivalent to $(G, -)$ and $G$ is a graph of odd girth at least $k$. For even values of $k$, the members of $\mathcal{SP}_{k}$ are signed bipartite planar graphs of negative girth at least $k$. With this terminology Conjecture~\ref{conj:PlanarToSPC(k)} can be restated as:

\begin{conjecture}[Conjecture~\ref{conj:PlanarToSPC(k)} restated]\label{conj:PlanarToSPC(k)-restated}
The class $\mathcal{SP}_{k}$ of signed planar graphs is bounded by $SPC(k-1)$. 
\end{conjecture}

We now propose the following conjecture and show that, if true, while it clearly generalizes Conjecture~\ref{conj:PlanarToSPC(k)}, it is in fact equivalent to it.

\begin{conjecture}\label{conj:SP_kBound->SP_k+1Bound}
If $\mathcal{SP}_{k}$ is bounded by a signed graph $\hat{B}$, then  $\mathcal{SP}_{k+1}$ is bounded by $EDC(\hat{B})$.
\end{conjecture}

Following Definition~\ref{def:SPC as EDC} of signed projective cubes, and noting that $\mathcal{SP}_{2}$ is clearly bounded by $SPC(1)$, it is clear that Conjecture~\ref{conj:SP_kBound->SP_k+1Bound} contains Conjecture~\ref{conj:PlanarToSPC(k)-restated}. In the following we show that the inverse implication also holds, and thus the two conjectures are equivalent. 

\begin{theorem}\label{thm:B->EDC(B)}
Suppose $SPC(i)$ bounds $\mathcal{SP}_{i+1}$ for every $i$, $1\leq i\leq k$. Let $\hat{B}$ be a signed graph which bounds the class $\mathcal{SP}_{l}$, for some $l$, $l\leq k$. Then $EDC(\hat{B})$ bounds the class $\mathcal{SP}_{l+1}$.
\end{theorem}

This implies, in particular, that, if for every positive integer $k$ Conjecture~\ref{conj:SP_kBound->SP_k+1Bound} holds for the choice of $\hat{B}=SPC(k)$, then it holds for every choice of $\hat{B}$ that satisfies the condition of the conjecture. 

\begin{proof}
Given a signed graph $(G, \sigma)$ in $\mathcal{SP}_{l+1}$, our goal is to show that it maps to $EDC(\hat{B})$. Since $(G, \sigma)$ maps to $SPC(l)$, by Theorem~\ref{thm:SPC(k)-packing}, we have signatures $\sigma_i$, $i=1,\ldots, l+1$, which pairwise do not share a negative edge and each is equivalent to $\sigma$. Let $E_i$, $i=1, \ldots, l+1$, be the set of negative edges of $(G, \sigma_i)$.

Let $G^{*}$ be the graph obtained from $G$ by contracting all the edges in $E_{l+1}$, where we delete the contracted edges, but resulting parallel edges remain in $G^*$. With a minor abuse of notation, let $\sigma_i$, $i=1, \ldots, l$, be the signature on $G^*$ which assigns a ``$-$'' to the edges in $E_i$ of $G^*$. 
In other words, while the signature $\sigma_{l+1}$ is eliminated, the other $l$ signatures remain untouched. 

In what follows we conclude two facts: 1. That $(G^{*}, \sigma_i)$ and $(G^{*}, \sigma_j)$ are switching equivalent for any pair $i, j \in\{1, 2,\ldots, l\}$. 2. That $(G^{*}, \sigma_1)\in \mathcal{SP}_l$. To that end, given a cycle $C^{*}$ of $G^{*}$, we show that first of all it has the same sign in each of $(G^{*}, \sigma_i)$, $i=1,2, \ldots, l$, and, secondly, that if this sign is positive then $C^{*}$ is of even length and that all the negative ones are of the same parity.

Let $C$ be a cycle of $G$ which is contracted to $C^{*}$. Let $s$ be the number of negative edges of $C$ in $(G, \sigma_{l+1})$, noting that $s=0$ is a possibility. If $C^{*}$ is a positive cycle in say $(G^{*},\sigma_i)$ for some $i$, $i \leq l$, then $C$ is a positive cycle in $(G,\sigma_i)$ and, hence, it is a positive cycle in each of $(G, \sigma_j)$, $j=1,2, \ldots, l+1$. There are two conclusions from this fact: the first is that $C$ is of an even length because $(G,\sigma)\in \mathcal{SP}_{l+1}$, and the second is that $s$ is even number. In other words, $C^{*}$ is obtained from an even cycle $C$ by contracting an even number of edges. Thus $C^{*}$ is an even cycle as well. Similarly, if $C^{*}$ is a negative cycle of some $(G^{*}, \sigma_i)$, then the cycle $C$ of $G$ which is contracted to $C^{*}$ is a negative cycle in $(G, \sigma_{i})$, but then its parity is different from that of $C^{*}$ as $s$ is odd. Since all negative cycles in $(G, \sigma)$ are of the same parity as $l+1$, the negative cycles in $(G^{*}, \sigma_1)$ are of the same parity as $l$, and, moreover, each negative cycle having at least one negative edge in each of $(G^{*}, \sigma_i)$, $i=1, 2, \ldots, l$, is of length at least $l$. In other words: $(G^{*}, \sigma_1)\in \mathcal{SP}_{l}$.

Thus, by our assumption, $(G^*, \sigma_1)$  maps to $\hat{B}$. Let $\sigma'$ be the signature equivalent to $\sigma_1$ under which there is an edge-sign-preserving homomorphism $\phi$ of $(G^*, \sigma')$ to $\hat{B}$. To complete the proof, based on $\phi$, we will build the mapping $\varphi$ as an edge-sign-preserving homomorphism of $(G, \sigma_{l+1})$ to $EDC(\hat{B})$. Once again, with a bit abuse of notation, we let $(G, \sigma')$ be the signed graph whose negative edges are the negative edges of $(G^{*}, \sigma')$. Observe that the image $C^{*}$ of a cycle $C$ of $G$ under the contraction of edges in $E_{l+1}$ is a closed walk in $G^{*}$. Since all contracted edges $(E_{l+1})$ are positive in $(G, \sigma_1)$, we have $\sigma_{1}(C^{*})=\sigma_{1}(C)$. Since $(G^{*}, \sigma')$ is switching equivalent to $(G^{*}, \sigma_1)$, we have $\sigma'(C^{*})=\sigma_{1}(C)$. From the equivalence of $(G,\sigma_1)$ and $(G, \sigma_{l+1})$ we conclude that: $\sigma'(C)=\sigma_{l+1}(C)$. As $C$ is an arbitrary cycle, this implies that $(G,\sigma')$ and $(G, \sigma_{l+1})$ are switching equivalent. We observe furthermore that, since all negative edges of $(G,\sigma_{l+1})$ are contracted in order to get $G^*$ from $G$, and as the negative edges of $(G,\sigma')$ are lifted from those of $(G^{*},\sigma')$, there is no common negative edge between $(G, \sigma')$ and $(G,\sigma_{l+1})$.

Since $\sigma_{l+1}$ is switching-equivalent to $\sigma'$, and since they have no common negative edge, we must have an edge-cut $(A,A^\complement)$ whose edges are precisely those edges that are either negative in $(G, \sigma')$ or in $(G, \sigma_{l+1})$. If there are more than one possibility for choosing $A$ (and $A^\complement=V(G)\setminus A$), we take an arbitrary one. 

Recall that for each vertex $x$ of $\hat{B}$ there are two vertices $x_0$ and $x_1$ in $EDC(\hat{B})$. 
We now define a mapping of $(G,\sigma_{l+1})$  to $EDC(\hat{B})$ as follows. For a vertex $v$ of $(G, \sigma)$, let $v^*$ be its image in $(G^*,\sigma_1)$ and suppose $\phi(v^*)=x$ where $x$ is a vertex of $\hat{B}$.  If the vertex $v$ is in part $A$ of $G$, then define $\varphi(v)=x_0$. If $v$ is in part $A^\complement$, then define $\varphi(v)=x_1$. We claim that $\varphi$ is an edge-sign preserving homomorphism of $(G, \sigma_{l+1})$ to $EDC(\hat{B})$.  

For a negative edge $uv$ of $(G, \sigma_{l+1})$ we have $u^*=v^*$ (recall that negative edges of $(G, \sigma_{l+1})$ are contracted in order to form $G^*$). On the other hand, for $u$ and $v$ to be adjacent by a negative edge in $(G, \sigma_{l+1})$, one must be in $A$ and the other in $A^\complement$. Thus, if $\phi(v^*)=x$, then either $\varphi(v)=x_0$ and $\varphi(u)=x_1$ or vice versa. In both cases the negative edge $uv$  of $(G, \sigma_{l+1})$ is mapped (under $\varphi$) to a negative edge of $EDC(\hat{B})$. 

For positive edges of $(G, \sigma_{l+1})$ we consider two possibilities: I. $uv$ is positive in $(G, \sigma'$). II. $uv$ is negative in $(G, \sigma')$. 

In case $I$, $u$ and $v$ are either both in $A$ or both in $A^\complement$, by symmetry assume they are both in $A$. Since $u^*v^*$ is a positive edge of $(G^*, \sigma')$ and since $\phi$ is an edge-sign-preserving mapping, $\phi(u^*)\phi(v^*)$ is a positive edge of $\hat{B}$. Assume $\phi(u^*)=x$ and $\phi(v^*)=y$, then, by the definition of $EDC(\hat{B})$ and since $xy$ is a positive edge of $\hat{B}$, $x_0y_0$ (and $x_1y_1$) are positive edges of $EDC(\hat{B})$. Thus $\varphi(u)\varphi(v)=x_0y_0$ is a positive edge of $EDC(\hat{B})$.

In case $II$, the edge $uv$ is in the cut $(A,A^\complement)$. Without loss of generality, assume $u\in A$ and $v\in A^\complement$. Observe that $u^*v^*$ is also a negative edge of $(G^*, \sigma')$ and since $\phi$ is an edge-sign-preserving mapping of $(G^*,\sigma')$ to $\hat{B}$, $\phi(u^*)\phi(v^*)$ is a negative edge of $\hat{B}$. Assume $\phi(u^*)=x$ and $\phi(v^*)=y$. Then by the definition of $\varphi$, we have $\varphi(u)=x_0$ and $\varphi(y)=y_1$. As $xy$ is a negative edge of $\hat{B}$, in $EDC(\hat{B})$ we have two positive edges $x_0y_1$ and $x_1y_0$. On the other hand, as $uv$ is a positive edge of $(G, \sigma_{l+1})$, the mapping $\varphi$ preserves adjacency of $u$ and $v$ and the sign of the $uv$ edge with respect to the signature $\sigma_{l+1}$. That is as we desired.
\end{proof}

It has been verified that $SPC(i)$ bounds $\mathcal{SP}_{i+1}$ for $i\leq 7$: the case of $i=2$ is the four-color theorem. Applying induction on $i$, and thus using the 4CT as a base of the induction, the cases $i=3,4$ are proved in \cite{NY23} and \cite{NY24} (respectively) using the notion of packing signatures. The cases $i=5, 6, 7$ (again by induction on $i$) are implied through results of \cite{DKK16}, \cite{CEKS15}, and \cite{CES15} and result of \cite{NRS13} which provides the relation with the conjecture and an edge-coloring conjecture of Seymour. 

Hence, starting with the Gr\"otzsch theorem we have several corollaries. The first is a result of \cite{NW23}, we shall note that Theorem~\ref{thm:B->EDC(B)} is an extension of this result. We note that $(K_{3,3}, M)$ is a signed graph on $K_{3,3}$ where the edges of a perfect matching are assigned negative sign, the rest being positive. 

\begin{corollary}\label{SP_6<K33M}
The class $\mathcal{SP}_{6}$ is bounded by $(K_{3,3}, M)$.
\end{corollary}

We should note our proof of this corollary, and the original proof in \cite{NW23}, relies on the proof of Conjecture~\ref{conj:PlanarToSPC(k)-restated} for $k=6$ which in turn is based on the proofs for $k=5$, $k=4$ and $k=3$, noting that the case $k=3$ is the 4CT. Based on a recent work, \cite{BHNSZ23}, this corollary can be proved independent of the 4CT. 

A signed graph $\hat{G}$ is said to be $r$-critical if $\chi_c(\hat{G})>r$ but every proper (signed) subgraph of it admits a circular $r$-coloring. It is proved in \cite{BHNSZ23} that every $3$-critical signed graph on $n$ vertices has at least $\frac{3n-1}{2}$ edges.  Toward a contradiction, assume that Corollary~\ref{SP_6<K33M} is not valid and that there is a signed bipartite graph of negative girth at least 6 which does not map to $(K_{3,3}, M)$. Moreover, assume among all such examples $(G, \sigma)$ is one with the smallest number of vertices and then the smallest number of edges.  
It is shown in \cite{NW23} that a signed bipartite graph admits a circular 3-coloring if and only if it maps to $(K_{3,3}, M)$. Since $(G, \sigma)$ does not map to $(K_{3,3}, M)$ and it is minimal for this property, it is a 3-critical signed graph. Thus it must have at least $\frac{3n-1}{2}$ edges. On the other hand, we claim that in any planar embedding of $(G,\sigma)$ every face of it must be a (negative) 6-cycle. That is by minimality of $(G, \sigma)$ and by the folding lemma of \cite{NRS13}. As such, by the Euler formula, we have $e=\frac{3n-6}{2}$, contradicting the lower bound on the number of edges.

Applying the same to this corollary, we get a (signed) graph on 12 vertices whose odd girth is 5 and bounds the class $\mathcal{SP}_{7}$. As in this case we may switch all signed graphs so that all edges are negative, the homomorphism problem is reduced to homomorphisms of graphs. Thus we may state the results in the language of graphs. To this end, and also because of an independent interest, we present a simplified version of $EDC(EDC(G, -))$ purely in the language of graphs. 

Given a graph $G$ and $4$-cycle $C$ on vertices $1,2,3,4$ in the cyclic order, we define $C*G$ to be the graph on vertex set $V(G)\times [4]$ whose edges are as follows. For each vertex $u$ of $G$, the four vertices $(u,1)$, $(u,2)$, $(u,3)$, and $(u,4)$ induce a 4-cycle in this cyclic order. For each edge $uv$ of $G$ the following four pairs form edges: $(u,1)(v,3)$, $(u,2)(v,4)$, $(u,3)(v,1)$, and $(u,4)(v,2)$. The products $K_2 * C_4$, $K_3 * C_4$ and $K_4 * C_4$ are presented in Figure~\ref{fig:K2C4}, Figure~\ref{fig:K3*C4}, and Figure~\ref{fig:K4*C4} respectively. One may observe that $K_4 * C_4$ is isomorphic to $PC(4)$ depicted differently in Figure~\ref{fig:SPC(4)}.

\begin{figure}[ht]
\centering
\resizebox{8cm}{5cm}{
\begin{tikzpicture}[thick,scale=1]
\foreach \i in {1,2,3,4}
{
\draw[rotate=90*\i] (1, 1) node[circle, fill=white, draw=black!80, inner sep=0mm, minimum size=2mm]  (A\i){$\i$};
}

\foreach \i in {1,2,3,4}
{
\draw[rotate=90*\i] (1, 1) node[circle, yshift=+5cm, fill=white, draw=black!80, inner sep=0mm, minimum size=2mm] (B\i){$\i$};
}

\draw (-7, 0) node[circle,  fill=white, draw=black!80, inner sep=0mm, minimum size=3.5mm] (u){$u$};

\draw (-7, 5) node[circle,  fill=white, draw=black!80, inner sep=0mm, minimum size=3.5mm] (v){$v$};  

\draw[very thick] (u) -- (v);  

\coordinate (O) at (0,2.5);

\draw[->] (-5,2.5) -- (-3, 2.5);

\begin{scope}[on background layer]

    \draw[dashed] (A1) -- (B3);
    \draw[dashed] (A2) -- (B4);
    \draw[dashed] (A3) -- (B1);
    \draw[dashed] (A4) -- (B2);

\draw [fill=lightgray, thick, opacity=.8] (A1.center) -- (A2.center) -- (A3.center) -- (A4.center) -- (A1.center);



\draw [fill=lightgray, thick, opacity=.8] (B1.center) -- (B2.center) -- (B3.center) -- (B4.center) -- (B1.center);

\end{scope}

\end{tikzpicture}}
\caption{$K_2 * C_4$}
\label{fig:K2C4}
\end{figure}

\begin{figure}[ht]
\centering 

\begin{minipage}{.48\textwidth}
\centering
\begin{tikzpicture}
\foreach \i in {1,...,15}
{
\draw[rotate=30*\i] (2.2, 2.2) node[circle, fill=white, draw=black!80, inner sep=0mm, minimum size=2mm]  (\i){};
}

\foreach \i in {1,...,12}
{
    
\draw (\i) -- (\the\numexpr\i+1);
}

\foreach \i in {1,...,6}
{
    
\draw (\i) -- (\the\numexpr\i+6);
}

\foreach \i in {1,3,...,11}
{
    
\draw (\i) -- (\the\numexpr\i+3);
}

\end{tikzpicture}
\end{minipage}
\begin{minipage}{.49\textwidth}
\centering
\begin{turn}{190}
\begin{tikzpicture}[scale=.5]

\foreach \i in {1,2,3,4}
{
\draw[rotate=90*\i+45] (1, 1) node[circle, fill=white, draw=black!80, inner sep=0mm, minimum size=2mm]  (A\i){};
}

\foreach \i in {1,2,3,4}
{
\draw[rotate=90*\i+45] (1, 1) node[circle, yshift=5cm, xshift=-2cm, fill=white, draw=black!80, inner sep=0mm, minimum size=2mm] (B\i){};
}

\foreach \i in {1,2,3,4}
{
\draw[rotate=90*\i+45] (1, 1) node[circle, yshift=+4cm, xshift=4cm, fill=white, draw=black!80, inner sep=0mm, minimum size=2mm] (C\i){};
}

\begin{scope}[on background layer]

    \draw (A1) -- (B3);
    \draw (A2) -- (B4);
    \draw (A3) -- (B1);
    \draw (A4) -- (B2);

    \draw (A1) -- (C3);
    \draw (A2) -- (C4);
    \draw (A3) -- (C1);
    \draw (A4) -- (C2);

    \draw (B1) -- (C3);
    \draw (B2) -- (C4);
    \draw (B3) -- (C1);
    \draw (B4) -- (C2);

\draw [fill=lightgray, opacity=.8] (A1.center) -- (A2.center) -- (A3.center) -- (A4.center) -- (A1.center);
\draw [fill=lightgray, opacity=.8] (B1.center) -- (B2.center) -- (B3.center) -- (B4.center) -- (B1.center);
\draw [fill=lightgray, opacity=.8] (C1.center) -- (C2.center) -- (C3.center) -- (C4.center) -- (C1.center);
\end{scope}
\end{tikzpicture}
\end{turn}
\end{minipage}

\caption{Two presentations of $K_3 * C_4$}
\label{fig:K3*C4}
\end{figure}
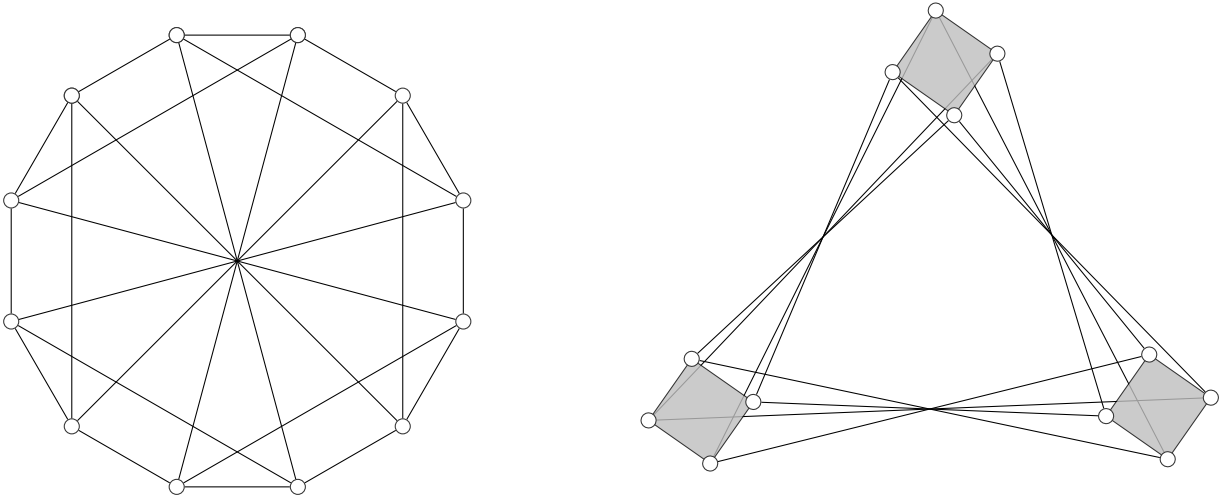

\begin{corollary}\label{cor:P7<K3*C4}
The class $\mathcal{P}_7$ of planar graphs of odd girth at least 7 is bounded by $K_3*C_4$.
\end{corollary}

We note that the graph $K_3*C_4$ is 3-colorable. It has independence number 4 because any independent set of size at least 4 must take at least one pair from one of the three $C_4$. Such a pair then must be non adjacent in the said $C_4$. As all such pairs are isomorphic, having taken such a pair, the set of vertices not adjacent to either induces $2K_2$ and hence at most two more vertices can be added to the independent set. As the graph is vertex-transitive, its fractional chromatic number is $\frac{12}{4}=3$. Thus its circular chromatic number, which is sandwiched between fractional chromatic number and the chromatic number, is also 3. Hence this triangle-free bound on 12 vertices for $\mathcal{P}_7$, while of independent interest, does not improve any of the known bounds on circular or fractional chromatic number of this class of planar graphs. 

\begin{figure}[ht]
\centering
\resizebox{8cm}{8cm}{
\begin{tikzpicture}[thick,scale=.4]

\foreach \i in {1,2,3,4}
{
\draw[rotate=90*\i+45] (3, 3) node[circle, yshift=-1cm, xshift=-1, fill=white, draw=black!80, inner sep=0mm, minimum size=2.5mm]  (A\i){};
}

\foreach \i in {1,2,3,4}
{
\draw[rotate=90*\i+45] (3, 3) node[circle, yshift=6cm, xshift=-6cm, fill=white, draw=black!80, inner sep=0mm, minimum size=2.5mm] (B\i){};
}

\foreach \i in {1,2,3,4}
{
\draw[rotate=90*\i+45] (3, 3) node[circle, yshift=7cm, xshift=7cm, fill=white, draw=black!80, inner sep=0mm, minimum size=2.5mm] (C\i){};
}

\foreach \i in {1,2,3,4}
{
\draw[rotate=90*\i+45] (3, 3) node[circle, yshift=12cm, xshift=1cm, fill=white, draw=black!80, inner sep=0mm, minimum size=2.5mm]  (D\i){};
}

\begin{scope}[on background layer]

    \draw (A1) -- (B3);
    \draw (A2) -- (B4);
    \draw (A3) -- (B1);
    \draw (A4) -- (B2);

    \draw (A1) -- (C3);
    \draw (A2) -- (C4);
    \draw (A3) -- (C1);
    \draw (A4) -- (C2);

    \draw (B1) -- (C3);
    \draw (B2) -- (C4);
    \draw (B3) -- (C1);
    \draw (B4) -- (C2);

    \draw (D1) -- (A3);
    \draw (D2) -- (A4);
    \draw (D3) -- (A1);
    \draw (D4) -- (A2);

    \draw (D1) -- (B3);
    \draw (D2) -- (B4);
    \draw (D3) -- (B1);
    \draw (D4) -- (B2);

    \draw (D1) -- (C3);
    \draw (D2) -- (C4);
    \draw (D3) -- (C1);
    \draw (D4) -- (C2);

\draw [fill=lightgray, opacity=.8] (A1.center) -- (A2.center) -- (A3.center) -- (A4.center) -- (A1.center);
\draw [fill=lightgray, opacity=.8] (B1.center) -- (B2.center) -- (B3.center) -- (B4.center) -- (B1.center);
\draw [fill=lightgray, opacity=.8] (C1.center) -- (C2.center) -- (C3.center) -- (C4.center) -- (C1.center);
\draw [fill=lightgray, opacity=.8] (D1.center) -- (D2.center) -- (D3.center) -- (D4.center) -- (D1.center);
\end{scope}

\end{tikzpicture}}
\caption{$K_4 * C_4$}
\label{fig:K4*C4}
\end{figure}

\begin{corollary}
	The signed bipartite graph $EDC(K_3 * C_4)$ bounds the class of signed bipartite graphs of negative girth at least 8.
\end{corollary}

We note that $EDC(K_3 * C_4)$ is signed bipartite graph on 24 vertices having negative girth 6.

\subsection{Equivalent conjectures}

As we noted, Conjecture~\ref{conj:PlanarToSPC(k)} is equivalent to several other conjectures. Here we mention a few with references for the connections.

The first is about packing signatures in signed graphs. An obvious upper bound for the signature packing number of a signed graph is its negative girth. When this upper bound is achieved, we says $(G,\sigma)$ \emph{packs}. For more on the following conjecture we refer to \cite{NWZ21}. We note that $\mathcal{C}_{10}$ is the class of signed graphs where every cycle is either positive and even or negative and even, i.e., the class of signed bipartite graphs. Similarly, $\mathcal{C}_{11}$ is the class of signed graph where every cycle is either positive and even or negative and odd. Each member of this class can be switched to have all edges negative.

\begin{conjecture}\label{conj:Packing}
Any signed planar graph in the subclass $\mathcal{C}_{10}\cup \mathcal{C}_{11}$ packs.
\end{conjecture}

Next to mention is a general conjecture of P. Seymour on the edge-chromatic number of planar multigraphs. Recall that fractional edge-chromatic number of a multigraph $G$, denoted $\chi'_f(G)$, is the solution to the linear program that is obtained by writing edge-chromatic number as an integer program and then allowing variable to take real values. Seymour has conjectured a strong connection between fractional edge-chromatic number and edge-chromatic number of planar multigraphs:

\begin{conjecture}\label{conj:SeymourEdgeColoring}
For every planar multigraph $G$ we have $\chi'(G)=\lceil \chi'_f(G)\rceil$.
\end{conjecture}

When $G$ is furthermore $k$-regular, then $\chi'_f(G)=k$ holds under a connectivity condition. The validity of Conjecture~\ref{conj:SeymourEdgeColoring} in this special case is strongly related to Conjecture~\ref{conj:Packing} (via duality) and, therefore, to Conjecture~\ref{conj:PlanarToSPC(k)}. For details see \cite{NRS13}, \cite{NY23} and references therein.

%

\subsection{Extension to minor closed families}

A minor of a signed graph $(G, \sigma)$ is a signed graph obtained by any sequence of the following operations:

\begin{itemize}
	\item Deleting a vertex or an edge,
	\item Switching at a vertex,
	\item Contracting a positive edge.
\end{itemize}

A class $\mathcal{C}$ of signed graphs is said to be \emph{minor closed} if for any member $(G,\sigma)$ of $\mathcal{C}$, any minor $(H,\pi)$ of $(G, \sigma)$ is in $\mathcal{C}$. 

As only positive edges can be contracted, sign of the image of a cycle, after contracting an edge of it, remains unchanged. This is the key difference between the notions of minor in signed graphs and graphs. For example, while the class of $K_3$-minor free graphs is the class of all forests, the class of $(K_3, -)$-minor free signed graphs is the class of all balanced signed graphs. So while the former is quite a sparse family of graphs, the latter contains for example $(K_n,+)$ for every $n$. 

We now may observe that in proving Theorem~\ref{thm:B->EDC(B)} the only real use of planarity was that after contracting negative edges of $(G,\sigma_{l+1})$ the result is still planar. We used the assumption that $(G, \sigma)$ maps to $SPC(l)$ to conclude that the contracted graph has the required girth condition,  following which based on a mapping of $(G^{*}, \sigma')$ to $\hat{B}$ we built a mapping of $(G, \sigma_{l+1})$. One may observe that edges negative in $(G, \sigma_{l+1})$ are positive in any of $(G, \sigma_i)$, $i\leq l$. Thus $(G^{*}, \sigma')$ is a minor of $(G,\sigma)$. In other words, if we have a minor closed family $\mathcal{C}$ of signed graphs for which the following conditions hold: for $l\leq k$, any signed graph $(G,\sigma)$ in $\mathcal{C}$ satisfying $g_{ij}(G, \sigma)\geq g_{ij}(SPC(l))=g_{ij}(C_{-l-1})$ maps to $SPC(l)$ and if there is a signed graph $\hat{B}$, $g_{ij}(\hat{B})=l\leq k-1$, for which we have: any signed graph $(G,\sigma)$ in $\mathcal{C}$ satisfying $g_{ij}(G, \sigma)\geq g_{ij}(SPC(l))=g_{ij}(\hat{B})$ maps to $\hat{B}$, then any signed graph $(G,\sigma)$ in $\mathcal{C}$ satisfying $g_{ij}(G, \sigma)\geq g_{ij}(SPC(l))=g_{ij}(EDC(\hat{B)})$ maps to $EDC(\hat{B})$.

The largest minor closed family of signed graphs for which we may expect this method to work is the class of $(K_5,-)$-minor free signed graphs. Indeed B. Guenin has conjectured in an unpublished manuscript that Conjecture~\ref{conj:PlanarToSPC(k)} holds if we replace planarity with no $(K_5,-)$-minor. If so, we may conclude, similarly, that whenever $g_{ij}(\hat{B})=g_{ij}(C_{-l})$ for some $l$ and $\hat{B}$ bounds the subclass of $(K_5,-)$-minor free signed graphs satisfying $g_{ij}(G, \sigma)\geq g_{ij}(C_{-k})$ for some $k$, then $EDC(\hat{B})$ bounds the subclass of $(K_5,-)$-minor free signed graphs satisfying $g_{ij}(G, \sigma)\geq g_{ij}(C_{-l-1})$.


\end{document}